\documentclass[11pt,a4paper]{amsart}

\usepackage{amssymb}
\usepackage{amsthm}
\usepackage[all]{xy}
\usepackage{graphicx}
\usepackage[utf8]{inputenc}

\usepackage{mathtools}
\mathtoolsset{showonlyrefs}

\usepackage{stix}

\usepackage{geometry}

\newcommand{\OO}{\mathcal O}

\newcommand{\Z}{\mathbb Z}

\newcommand{\C}{\mathbb C}

\newcommand{\id}{\operatorname{id}}

\newcommand{\Irr}{\operatorname{Irr}}

\newcommand{\Rad}{{\operatorname{Rad}}}

\newtheorem{defi}{Definition}[section]
\newtheorem{remark}[defi]{Remark}
\newtheorem{example}[defi]{Example}
\newtheorem{conjecture}[defi]{Conjecture}

\newtheorem{corollary}[defi]{Corollary}

\newtheorem{prop}[defi]{Proposition}

\usepackage{lipsum}

\usepackage[hidelinks]{hyperref}

\title{A counterexample to a conjecture on Cartan determinants of monoid algebras}

\author{Florian Eisele}
\address{Department of Mathematics, University of Manchester, Oxford Road, Manchester, M13 9PL}
\email{florian.eisele@manchester.ac.uk}

\newtheoremstyle{named}{}{}{\itshape}{}{\bfseries}{.}{.5em}{\thmnote{#3}}
\theoremstyle{named}

\renewcommand{\leq}{\leqslant}
\renewcommand{\geq}{\geqslant}


\begin{document}
\begin{abstract}
    We show that there are finite monoids $M$ such that the Cartan matrix of the monoid algebra $\C M$ is non-singular, whilst the Cartan matrix of $kM$ is singular for some field $k$ of positive characteristic, disproving a recent conjecture of Steinberg.
\end{abstract}

\maketitle

\section{Introduction}

It is well-known that, regardless of the base field, the Cartan matrix of a finite group algebra is always non-singular. In fact, if the base field has characteristic $p>0$, its determinant is always a $p$-power, and if the base field is $\C$, the Cartan matrix is simply the identity matrix. Steinberg \cite{SteinbergModRepMonoid} recently conjectured that this behaviour partially generalises to finite monoid algebras.

\begin{conjecture}[{\cite[Conjecture 3.7]{SteinbergModRepMonoid}}]\label{conjecture cartan}
    Let $M$ be a finite monoid, and assume that the Cartan matrix of the monoid algebra $\C M$ is non-singular. Then the Cartan matrix of $kM$ is non-singular for any field~$k$.
\end{conjecture}

We will provide a counterexample to this conjecture, and it is easy to produce further such counterexamples using the same construction. The conjecture was shown to be true in \cite{SteinbergModRepMonoid} for regular monoids and monoids with aperiodic (left or right) stabilisers. We look at the simplest construction of a finite monoid which fails to lie in either of these classes: take a monoid with a zero, let the maximal subgroup at $1$ be some finite group $G$, and put a number of $\mathcal J$-classes below $1$ which square to zero (making them non-regular) and on which $G$ acts highly non-trivially (breaking the aperiodic stabiliser condition). Since these monoids will be described entirely in group-theoretical terms, it is easy to compute their decomposition matrices and Cartan matrices. It remains an interesting question whether there is a theorem which generalises both of the positive results mentioned above. The counterexample constructed below would suggest that one would have to impose some restriction on the action of a maximal subgroup $G_e$ on $eMxMe$, where $e\in M$ is an idempotent, and $x\in M$ is below $e$ in the $\mathcal J$-order.

\section{Construction of a counterexample}

Throughout this section let $G$ be a finite group, and assume $X$ is a finite $(G,G)$-biset. We use ``$\uplus$'' to denote a disjoint union.

\begin{defi}
     Define a monoid
    \[
        M(G, X) = G \uplus X \uplus \{ z  \},
    \]
    where the multiplication within $G$ is just group multiplication and $z$ is a zero in the monoid. Furthermore, the products $gx$ and $xg$ for $g\in G$ and $x\in X$ are given by the biset structure on $X$, and 
    \[
        xy=z \quad \textrm{for all } x,y\in X.
    \] 
\end{defi}
For what follows note that we can view the biset $X$ as a (left) $G\times G$-set by decreeing $(g,h)\cdot x = gxh^{-1}$. We will switch back and forth between $(G,G)$-bisets and $G\times G$-sets.
Note that in the monoid $M(G,X)$ defined above, the non-regular $\mathcal J$-classes are exactly the $G\times G$-orbits on $X$.

\begin{remark}
    Let $k$ be a field.
    We will work exclusively with the contracted monoid algebras $k_0 M(G,X)=k M(G,X) / k z$. We have $k M(G,X) \cong k_0 M(G,X)\times k$, so the Cartan matrix of $k M(G,X)$ is non-singular if and only if  that of $k_0 M(G,X)$ is. In particular, we may replace the monoid algebras in Conjecture~\ref{conjecture cartan} by contracted monoid algebras without changing the conjecture.
\end{remark}

\begin{prop}
    Let $k$ be a field. Up to isomorphism, the simple $k_0M(G,X)$-modules are exactly the simple $kG = k_0M(G,X)/kX$-modules. 
\end{prop}
\begin{proof}
    By construction $kX$ is an ideal in $k_0M(G,X)$ satisfying $(kX)^2=0$. Hence $kX \subseteq \Rad(k_0M(G,X))$, which implies the assertion.
\end{proof}

The above also implies that, other than for the unique simple module with apex $z$, the decomposition numbers of $M(G,X)$ are identical to the decomposition numbers of $G$ with respect to any prime $p>0$. That gives us the following.

\begin{prop}
    Let $(K,\OO,k)$ be a $p$-modular system which is splitting for $M(G,X)$, for some prime $p>0$.
    We have 
    \begin{equation}\label{eqn cartan ddtr}
        C(k_0 M(G, X)) = D^\top \cdot C(K_0 M(G, X))\cdot  D,
    \end{equation}
    where $C(k_0 M(G, X))$ and $C(K_0 M(G, X))$ denote the respective Cartan matrices and $D$ denotes the decomposition matrix of $G$.
\end{prop}
\begin{proof}
    Equation~\eqref{eqn cartan ddtr} follows from the corresponding formula for the (non-contracted) monoid algebras, which is given in \cite[Theorem~3.5]{SteinbergModRepMonoid}. The fact that the decomposition matrix in this formula is equal to the decomposition matrix of $G$ follows from \cite[Proposition~3.3]{SteinbergModRepMonoid}.
\end{proof}

Of course the Cartan matrices over $K$ and over $\C$ coincide, and the crucial point of the above is that the decomposition matrix $D$ does not depend on $X$ at all. So the strategy is to vary $X$ in such a way as to make the Cartan matrix over $k$ singular.

\begin{prop}\label{prop cartan matrix}
    Let $L_1,\ldots,L_n$ be subgroups of $G\times G$, not necessarily all distinct. Let $X$ denote the $(G,G)$-biset corresponding to 
    \[
        \biguplus_{i=1}^n (G\times G)/L_i.
    \]
    Then the Cartan matrix of $\C_0 M(G,X)$ is given by
    \begin{equation}\label{eqn formula cartan}
          C_{\chi, \eta} = \delta_{\chi,\eta} + \sum_{i=1}^n (\chi\otimes \bar \eta)(\widehat L_i),
    \end{equation}
    where $\delta_{\chi,\eta}$ denotes the Kronecker delta, $\bar\eta$ is the complex conjugate of $\eta$, and $\widehat L_i = \frac{1}{|L_i|}\sum_{g\in L_i} g$.
\end{prop}
\begin{proof}
    For any $\chi\in \Irr_\C(G)$ let $e_\chi$ denote a primitive idempotent in $\C G$ such that $\C G e_{\chi}$ is a simple $\C G$-module with character $\chi$. By construction, the idempotent $e_\chi$ remains primitive in $\C_0 M(G,X)$, and the Cartan number $C_{\chi,\eta}$ is simply  the dimension of $e_\chi \C_0 M(G,X) e_\eta$. So 
    \[
        C_{\chi,\eta} = \dim e_\chi \C G e_\eta + \dim e_\chi \C X e_\eta =  \dim e_\chi \C G e_\eta + \sum_{i=1}^n \dim (e_\chi \otimes e_{\bar \eta} \cdot \C (G\times G/L_i)).
    \]
    Now, passing to characters and using Frobenius reciprocity gives 
    \[
        \begin{array}{rcl}
        \dim (e_\chi \otimes e_{\bar \eta} \cdot \C (G\times G/L_i)) 
        &=&  \dim (e_\chi \otimes e_{\bar \eta} \cdot {\rm Ind}_{L_i}^{G\times G} \mathbb C)\\
        &=& (\chi\otimes \bar \eta, {\rm Ind}_{L_i}^{G\times G} 1)_G = ({\rm Res}_{L_i}^{G\times G}(\chi\otimes \bar \eta),  1)_{L_i}\\
        &=& (\chi\otimes \bar \eta)(\widehat L_i).
        \end{array}
    \]
    Moreover, $\dim e_\chi \C G e_\eta = (\chi,\eta)_G =\delta_{\chi,\eta}$. The asserted formula now follows.
\end{proof}

\begin{defi}
    For a subgroup $L\leq G\times G$ define a matrix $\Delta(L)$ with rows and columns indexed by the elements of $\Irr_\C(G)$ as follows:
    \[
        \Delta(L)_{\chi,\eta}= (\chi\otimes \bar \eta)(\widehat L).
    \]
\end{defi}

With equations \eqref{eqn cartan ddtr} and \eqref{eqn formula cartan} in mind it is now fairly clear what we need to do to disprove Conjecture~\ref{conjecture cartan}. We need to find coefficients $z_L \in \Z_{\geq 0}$ such that 
\[ 
    \id + \sum_{L\leq G\times G} z_L \Delta(L)
\]
is non-singular and 
\[ 
    D^\top \cdot \left(\id + \sum_{L\leq G\times G} z_L \Delta(L) \right)\cdot D
\]
is singular, where $D$ is the decomposition matrix of $G$ with respect to some prime $p>0$. 
The following concrete example was found using \textsc{Gap}, but it can be verified by hand. Similar examples are easy to find in other finite groups. 
\begin{example}
    Let $G=S_3$, the symmetric group on three letters. Set 
    \[
        L_a = \{ (g,g^{-1}) \ | \ g \in G\}, \quad L_b = \{1\} \times \langle (1,2)\rangle, \quad L_c = G \times \langle (1,2,3) \rangle.
    \]
    The character table of $S_3$ and its decomposition matrix for $p=3$ are given by 
    \[
        \begin{array}{rccc}
        & () & (1,2) & (1,2,3)\\\hline
        \chi_{(3)} & 1 & 1 & 1\\
        \chi_{(2,1)} & 2 & 0 & -1\\
        \chi_{(1^3)} & 1 & -1 & 1
        \end{array}
        \quad \textrm{and}\quad
        \begin{array}{r|cc}
            & \psi_{(3)} & \psi_{(2,1)} \\\hline
            \chi_{(3)} & 1 & 0 \\
            \chi_{(2,1)} & 1& 1\\
            \chi_{(1^3)} & 0 & 1    
        \end{array},
    \]
    using the usual labelling of ordinary and modular characters of the symmetric groups.
    It is easy to see that $\Delta(L_a)$ is the $3\times 3$-identity matrix. As both $L_b$ and $L_c$ are of the form $H_1\times H_2$ for subgroups $H_1,H_2\leq G$ we can use the formula
    \[
        (\chi\otimes \bar \eta)(\widehat{H_1\times H_2}) = (\chi\otimes \bar \eta)(\widehat{H_1}\cdot \widehat{H_2}) = \chi(\widehat{H_1})\cdot \bar{\eta}(\widehat{H_2}) 
    \]
    to simplify evaluation of the entries of $\Delta(L_b)$ and $\Delta(L_c)$. We obtain
    \[
        \Delta(L_b) = \left( \begin{array}{ccc} 1 & 1 & 0 \\ 2 & 2& 0\\ 1 & 1 & 0  \end{array} \right)
        \quad \textrm{and}\quad
        \Delta(L_c) = \left( \begin{array}{ccc} 1 & 0 & 1\\ 0&0&0\\0&0&0 \end{array} \right).
    \]
    Now set $L_1=\ldots=L_4 = L_a$ (four copies), $L_5=L_6=L_b$ (two copies) and $L_7=\ldots=L_{171}=L_c$ (165 copies), and define $X$ as in Proposition~\ref{prop cartan matrix}. Then the Cartan matrices become
    \[
        C(\C_0 M(G,X)) = \left(
            \begin{array}{ccc}
                 172 & 2 & 165 \\4& 9& 0 \\ 2&2& 5
            \end{array}
        \right),
    \]
    which is non-singular, 
    and over a field $k$ of characteristic three 
    \[
        C(k_0 M(G,X)) = \left(\begin{array}{ccc} 1 & 1 & 0\\ 0 & 1 & 1\end{array} \right)\cdot \left(
            \begin{array}{ccc}
                 172 & 2 & 165 \\4& 9& 0 \\ 2&2& 5
            \end{array}
        \right)
        \cdot \left(\begin{array}{cc} 1 & 0 \\ 1 & 1\\ 0 & 1 \end{array} \right)
        = \left(
            \begin{array}{cc}
                187 & 176 \\ 17 & 16 
            \end{array}
         \right),
    \]
    which is singular.
\end{example}

\begin{corollary}
    Conjecture~\ref{conjecture cartan} does not hold true in general.
\end{corollary}

\bibliographystyle{alpha}
\bibliography{refs}

\begin{thebibliography}{Ste23}

\bibitem[Ste23]{SteinbergModRepMonoid}
B.~Steinberg.
\newblock The modular representation theory of monoids and a conjecture on the
  cartan determinant of a monoid algebra.
\newblock 2023.
\newblock {preprint (\url{https://arxiv.org/abs/2305.08251v2})}.

\end{thebibliography}

\end{document}